\documentclass[11pt]{article}
\usepackage[utf8]{inputenc}
\usepackage{graphicx} 
\usepackage{amsmath,amsfonts,amssymb,amsthm}
\usepackage[mathscr]{eucal}
\usepackage{amscd}
\usepackage{tikz}
\usepackage{tkz-graph}
\usepackage{multicol}
\usepackage{float}
\usepackage{authblk}

\newtheorem{theorem}{Theorem}[section]
\newtheorem{lemma}[theorem]{Lemma}
\newtheorem{definition}[theorem]{Definition}
\newtheorem{example}[theorem]{Example}
\newtheorem{proposition}[theorem]{Proposition}

\newtheorem{remark}[theorem]{Remark}

\title{Counting Polynomial-type Exceptional Units on Algebraic Varieties over Number Fields}
\author{Chen Lin$^1$, Kaihan Tang$^2$\footnote{Corresponding author.}}
\affil{{\small {$^{1,2}$School of Mathematics, Nanjing University, Nanjing 210093, China}\\
$^1$chen.lin@smail.nju.edu.cn, 
$^2$kaihantang@smail.nju.edu.cn} }
\date{}

\begin{document}
\maketitle
\begin{abstract}
Previous research on exceptional units has primarily focused on the ring of rational integers or abstract finite rings, often restricted to linear or quadratic constraints. In this paper, we extend the concept of polynomial-type exceptional units to the ring of integers of an arbitrary algebraic number field. We investigate the number of these polynomial-type exceptional units on general algebraic varieties. By employing the Chinese Remainder Theorem and Hensel’s lifting technique, we derive an exact counting formula for the number of these exceptional units on a smooth closed subscheme under the assumption of good reduction. Furthermore, using the Lang-Weil inequality, we establish an asymptotic estimate for the counting function. In particular, we prove that for varieties of degree at most two, the error term can be significantly improved, yielding a sharper asymptotic bound.
\end{abstract}
  \noindent { 2020 \it Mathematics Subject Classification: 11D79, 14G15.} \\[1mm]
    \noindent {\it Keywords: exceptional units, ring of integers, scheme, Hensel's technique, Lang-Weil inequality}

\section{Introduction}
The concept of an \textit{exceptional unit} (\textit{exunit} for abbreviation), defined as a unit $u$ in a ring $R$ such that $1-u$ is also a unit, was first introduced by Nagell \cite{Nagell} in 1969 to study cubic Diophantine equations. Since then, it has found profound applications in algebraic number theory, most notably in Lenstra’s work on Euclidean number fields \cite{Lenstra}. Apart from these applications, the counting function of exunits itself possesses a rich arithmetic structure. Harrington and Jones \cite{HJ} investigated the dynamical behavior of the iterations of this counting function over $\mathbb{Z}/n\mathbb{Z}$. 

A number of works focused on counting the number of exunits satisfying certain constraints. Sander \cite{Sander} obtained the formula for the number of solutions in $\mathbb{Z}/n\mathbb{Z}$ to the linear congruence $x_1+x_2\equiv c\pmod{n}$ with both $x_1$ and $x_2$ being exunits. This was later extended to the sums of $k$ exunits by Yang and Zhao \cite{YZ}, and further generalized to arbitrary finite commutative rings by Miguel \cite{Miguel}. Beyond linear sums, there have been attempts to address non-linear constraints: Mollahajiaghaei \cite{Mo} derived the formula for the number of solutions to the quadratic congruence $x_1^2+x_2^2+\dots+x_k^2\equiv c\pmod{n}$ with each $x_i$ being an exunit, while Feng and Hong \cite{FH2023} investigated $x_1^e+x_2^e+\dots+x_k^e\equiv c\pmod{n}$ with a given positive integer $e$.

Anand, Chattopadhyay and Roy \cite{ACR2020} generalized this notion to ``$f$-exunits". For a polynomial $f(X)\in\mathbb{Z}[X]$ and an integer $n\geq 2$, an element $a\in\mathbb{Z}/n\mathbb{Z}$ is said to be an \textit{$f$-exunit} if $\mathrm{gcd}(f(a),n)=1$. This generalization naturally prompted several investigations into the counting functions of $f$-exunits similar to that of classical exunits as is mentioned above, see \cite{ACR2020,ACR2021,ZHZ,FH}. However, unlike the classical case, they only focused on the sumsets of $f$-exunits.

Previous studies have primarily focused on the ring of rational integers $\mathbb{Z}$ or abstract finite rings, with investigations largely restricted to linear constraints (i.e. sumsets). In his study of finite rings, Miguel \cite{Miguel} explicitly pointed out that these counting problems can be extended to the ring of integers of an algebraic number field. Another natural and significant question is: instead of a single linear equation, can we determine the counting formula of $f$-exunits on a general polynomial equation or algebraic varieties defined by systems of polynomial equations? 

In this paper, we extend these results to the setting of number field. Our main result (Theorem \ref{main}) provides a counting formula for $f$-exunits lying on a variety $X$ modulo an ideal $\mathfrak{n}$.

To state our main results precisely, we first generalize the definition of $f$-exunits to the setting of an algebraic number field.

\begin{definition}
    Let $K$ be an algebraic number field with $\mathcal{O}_K$ its ring of integers and $\mathfrak{n}$ a non-zero integral ideal of $\mathcal{O}_K$. For a non-constant polynomial $f(X)\in\mathcal{O}_K[X]$, an element $\alpha\in \mathcal{O}_K/\mathfrak{n}$ is called an \textbf{$f$-exunit modulo $\mathfrak{n}$} or an \textbf{$(f,\mathfrak{n})$-exunit} if $f(\alpha)$ is a unit in the ring $\mathcal{O}_K/\mathfrak{n}$.  The set of all $(f,\mathfrak{n})$-exunits in $\mathcal{O}_K/\mathfrak{n}$ is denoted by $\mathscr{E}_{f,\mathfrak{n}}$.
\end{definition}

With the definition  of $f$-exunits extended to the ring of integers, we now turn to the counting problem for solutions satisfying a system of polynomial equations. We use the language of schemes, which simplifies the statement and proof of our main theorem. 

For a closed subscheme $X$ of $\mathbb{A}^n_{\mathcal{O}_K}$ (the $n$-dimensional affine space over $\mathcal{O}_K$), we define  the following set:
    $$\mathscr{E}_{f,\mathfrak{n}}(X)=X(\mathcal{O}_K/\mathfrak{n} )\cap \prod_{i=1}^n \mathscr{E}_{f,\mathfrak{n}},$$
where $X(\mathcal{O}_K/\mathfrak{n})$ denotes the $\mathcal{O}_K/\mathfrak{n}$-points of $X$.

Our main result provides an explicit counting formula for the cardinality of this set when the close subscheme $X$ satisfies smoothness conditions. 

\begin{theorem}
\label{main}
    Given a number field $K$ with ring of integers $\mathcal{O}_K$, let $f\in \mathcal{O}_K[x]$ be a polynomial, and $X$ be a smooth closed subscheme of $\mathbb{A}^n_{\mathcal{O}_K}$ with $\mathrm{codim}(X)=d$. For a nonzero integral ideal $\mathfrak{n}$ of $\mathcal{O}_K$ such that $X$ has a good reduction at $\mathfrak{p}$ for every prime ideal $\mathfrak{p}\vert \mathfrak{n}$, we have
 $$\#\mathscr{E}_{f,\mathfrak{n}}(X)=(\mathrm{N}_{K/\mathbb{Q}}(\mathfrak{n}))^{n-d}\prod_{\mathfrak{p}\vert \mathfrak{n}}\left(\frac{\#X(\mathcal{O}_K/\mathfrak{p})-\# \mathcal{N}^f(\mathfrak{p},X)}{(\mathrm{N}_{K/\mathbb{Q}}(\mathfrak{p}))^{n-d}}\right),$$
 where $\mathcal{N}^f(\mathfrak{p},X)=\{(x_1,\dots,x_n)\in X(\mathcal{O}_K/\mathfrak{p}):\prod_{i=1}^{n}\Bar{f}(x_i)=0\}$, $\mathrm{N}_{K/\mathbb{Q}}$ denotes the norm map in $K/\mathbb{Q}$, $\Bar{f}$ is the image of $f$ under the projection  $\mathcal{O}_K\longrightarrow \mathcal{O}_K/\mathfrak{p}$.
\end{theorem}

While Theorem \ref{main} establishes a formula for $\#\mathcal{E}_{f,\mathfrak{n}}(X)$, it relies on the local point counts $\#X(\mathcal{O}_K/\mathfrak{p})$ and the intersection sizes $\#\mathcal{N}^f(\mathfrak{p},X)$, which vary irregularly with the prime $\mathfrak{p}$. Using the Lang-Weil inequality, we can derive an bound for $\#\mathcal{E}_{f,\mathfrak{n}}(X)$ which depends only on the geometric invariants of $X$ and the absolute norm of the ideal $\mathfrak{n}$. 

\begin{theorem}
\label{approximate}
 Given a number field $K$ with ring of integers $\mathcal{O}_K$, let $f\in \mathcal{O}_K[x]$ be a polynomial. Let $X$ be a smooth closed subvariety of $\mathbb{A}^n_{\mathcal{O}_K}$ with $\mathrm{codim}(X)=d$, which is geometrically irreducible over K. For a nonzero integral ideal $\mathfrak{n}$ of $\mathcal{O}_K$ such that $X$ has a good reduction at $\mathfrak{p}$ for any prime ideal $\mathfrak{p}\vert \mathfrak{n}$, we have
    $$\#\mathscr{E}_{f,\mathfrak{n}}(X)=(\mathrm{N}_{K/\mathbb{Q}}(\mathfrak{n}))^{n-d}\exp\left({O\left(\frac{\log\mathrm{N}_{K/\mathbb{Q}}(\mathfrak{n})}{\log\log\mathrm{N}_{K/\mathbb{Q}}(\mathfrak{n})}\right)}\right).$$
    The implied constant in the $O$-terms depends only on $K,X$ and $f$.
\end{theorem}

Although Theorem \ref{approximate} provides an asymptotic bound for general smooth varieties, the error term can be significantly improved when the variety $X$ has low degree. Since previous studies have focused mostly on linear and quadratic congruences as discussed before, we derive the following sharper estimate for the varieties of degree no more than 2. 

\begin{theorem}\label{approximate 2}
     Given a number field $K$ with ring of integers $\mathcal{O}_K$, let $f\in \mathcal{O}_K[x]$ be a polynomial. Let $X$ be a smooth closed subvariety of $\mathbb{A}^n_{\mathcal{O}_K}$ with $\mathrm{codim}(X)=d$ and $\mathrm{deg}X\le 2$, which is geometrically irreducible over K. For a nonzero integral ideal $\mathfrak{n}$ of $\mathcal{O}_K$ such that $X$ has a good reduction at $\mathfrak{p}$ for any prime ideal $\mathfrak{p}\vert \mathfrak{n}$, we have$$\#\mathscr{E}_{f,\mathfrak{n}}(X)=\mathrm{N}_{K/\mathbb{Q}}(\mathfrak{n})^{n-d}\left(
\log \mathrm{N}_{K/\mathbb{Q}}(\mathfrak{n})\right)^{O(1)}.$$
The implied constant in the $O$-terms depends only on $K,X$ and $f$.
\end{theorem}

Our approach relies on the multiplicative structure of the counting function derived from the Chinese Remainder Theorem in Dedekind domains. For the counting functions in the local case, we employ a Hensel-type technique to lift solutions from the residue field $\mathcal{O}_K/\mathfrak{p}$ to $\mathcal{O}_K/\mathfrak{p}^k$. To establish the asymptotic estimate for $\#\mathscr{E}_{f,\mathfrak{n}}(X)$, we use the well-known Lang-Weil inequality to control the error terms arising from the local points on the variety.

This paper is organized as follows.

In Section \ref{proof of main}, we provide the proof of Theorem \ref{main}. The argument begins with establishing the multiplicativity of the counting function $\#\mathscr{E}_{f,\mathfrak{n}}(X)$ via the Chinese Remainder Theorem. Then we derived these local counting formulas by using Hensel lifting. To illustrate the concrete application of our counting formula, we present a detailed example involving a given variety $X$ over the quadratic field $\mathbb{Q}(\sqrt{-5})$ at the end of this section.

In Section \ref{proof of approximation}, we focus on the asymptotic behavior of $\#\mathscr{E}_{f,\mathfrak{n}}(X)$ as the norm of the ideal $N(\mathfrak{n})$ tends to infinity. Using the Lang-Weil inequality, we derive a general estimate for the number of $(f,\mathfrak{n})$-exunits, which depends on the dimension and degree of the variety and the absolute norm of $\mathfrak{n}$. Furthermore, we investigate the specific case where $\deg X \le 2$. Noting that the leading error term coefficient vanishes for quadratic varieties and linear varieties, we obtain a refined estimate that improves the error bound from $\exp\left({O\left(\frac{\log\mathrm{N}_{K/\mathbb{Q}}(\mathfrak{n})}{\log\log\mathrm{N}_{K/\mathbb{Q}}(\mathfrak{n})}\right)}\right)$ to $\exp\left(
O(\log\log \mathrm{N}_{K/\mathbb{Q}}(\mathfrak{n}))
\right)=\left(
\log \mathrm{N}_{K/\mathbb{Q}}(\mathfrak{n})\right)^{O(1)}$.


\section{Proof of Theorem \ref{main}}
\label{proof of main}

In this section, we complete the proof of Theorem \ref{main} by employing the Euler decomposition and the Hensel lifting technique. Moreover, we present a concrete example to illustrate the application of our counting formula.

Throughout this paper, for any $f \in \mathcal{O}_K[x_1,\dots,x_m]$, we denote by $\bar{f} \in (\mathcal{O}_K/\mathfrak{n})[x_1,\dots,\allowbreak x_m]$ its image under the reduction map modulo $\mathfrak{n}$. For brevity, we shall write $f(\alpha)$ to denote $\bar{f}(\alpha)$ for $\alpha \in (\mathcal{O}_K/\mathfrak{n})^m$.

First, we prove that $\#\mathscr{E}_{f,\mathfrak{n}}(X)$ is a multiplicative function with respect to the ideals of $\mathcal{O}_K$.

\begin{lemma}\label{decomposition}
    Let $\mathfrak{m},\mathfrak{n}$ be integral ideals of $\mathcal{O}_K$ satisfying $\mathfrak{m}+\mathfrak{n}=\mathcal{O}_K$, and $X$ be a subscheme of $\mathbb{A}_{\mathcal{O}_K}^n$. Then we have
    $$\mathscr{E}_{f,\mathfrak{mn}}(X)\cong \mathscr{E}_{f,\mathfrak{m}}(X)\times \mathscr{E}_{f,\mathfrak{n}}(X).$$
\end{lemma}
\begin{proof}
    According to the Chinese Remainder Theorem, we have a canonical ring isomorphism$$\mathcal{O}_K/\mathfrak{mn}\cong \mathcal{O}_K/\mathfrak{m}\times \mathcal{O}_K/\mathfrak{n}.$$

Since the functor of points of a scheme is compatible with finite products of rings,
for any subscheme of $\mathbb{A}^n_{\mathcal{O}_K}$, we have the following bijection $$X(\mathcal{O}_K/\mathfrak{mn})\cong X(\mathcal{O}_K/\mathfrak{m})\times X(\mathcal{O}_K/\mathfrak{n}).$$ 
Under this identification, the condition that
$f(x_i)$ be a unit in $\mathcal O_K/\mathfrak{mn}$ is equivalent to the condition
that its images in both $\mathcal O_K/\mathfrak m$ and
$\mathcal O_K/\mathfrak n$ are units.
Therefore, the above bijection restricts to the following bijection
$$
\mathscr E_{f,\mathfrak{mn}}(X)
\;\cong\;
\mathscr E_{f,\mathfrak m}(X)\times \mathscr E_{f,\mathfrak n}(X),
$$
which proves the lemma.
\end{proof}
\begin{remark}
    We note that the smoothness and good reduction assumptions imposed in Theorem \ref{main} are not required for this lemma. The multiplication property established here is a purely ring-theoretic consequence of the Chinese Remainder Theorem.
\end{remark}
Let $\mathfrak{n}=\mathfrak{p}_1^{\alpha_1}\dots \mathfrak{p}_r^{\alpha_r}$ be an ideal of $\mathcal{O}_K$ with $\mathfrak{p}_1,\dots,\mathfrak{p}_r$ distinct prime ideals of $\mathcal{O}_K$. By Lemma \ref{decomposition}, we have $\mathscr{E}_{f,\mathfrak{n}}(X)\cong\prod_{i=1}^r \mathscr{E}_{f,\mathfrak{p}_i^{\alpha_i}}(X)$. The following lemma addresses the lifting problem in the case where $\mathfrak{n}$ is a power of a prime ideal.
\begin{lemma}\label{lifing}
    Let $X$ be a smooth closed subscheme of $\mathbb{A}^n_{\mathcal{O}_K}$ with $\mathrm{codim}X=d$, $\mathfrak{p}$ be a prime ideal of $\mathcal{O}_K$ such that $X$ has a good reduction at $\mathfrak{p}$. For $\mathbf{a}\in X(\mathcal{O}_K/\mathfrak{p}^k)$, the set $$\{\mathbf{b}\in X(\mathcal{O}_K/\mathfrak{p}^{k+1}): \mathbf{b}\equiv \mathbf{a}\pmod{\mathfrak{p}^k}\}$$
    has cardinality $(\mathrm{N}_{K/\mathbb{Q}}(\mathfrak{p}))^{n-d}$.
\end{lemma}
\begin{proof}
    Let $\hat{\mathcal{O}}_K$ be the completion of $\mathcal{O}_K$ at $\mathfrak{p}$. Then $\hat{\mathcal{O}}_K$ is a discrete valuation ring with uniformizer $\pi$. We have $\hat{\mathcal{O}}_K/(\pi^m\hat{\mathcal{O}}_K)=\mathcal{O}_K/\mathfrak{p}^m$ for all positive integers $m$. Hence all counting statements modulo $\mathfrak{p^m}$ can be carried out after replacing
$\mathcal{O}_K$ by its $\mathfrak{p}$-adic completion.
 Without loss of generality, we may assume that $\mathcal{O}_K$ is complete at $\mathfrak{p}$, and hence a discrete valuation ring. Let $\pi$ denote the uniformizer of $\mathcal{O}_K$, i.e. $(\pi)=\mathfrak{p}$.

Let $I$ denote the ideal of $X$. Since $O_K[x_1,\dots,x_n]$ is Noetherian, the ideal $I$ is finitely generated. Let $\{g_1,\dots,g_m\}$ be its set of generators.

    All points of $X(\mathcal{O}_K/\mathfrak{p}^{k+1})$ lifting $\mathbf{a}$ have the form $\mathbf{\alpha} \equiv \mathbf{a}+\mathbf{\delta}\pi^k\pmod{\pi^{k+1}}$ for some $\delta \in ({\mathcal{O}}_K/(\pi))^n$ because $\pi^k\delta$ is well-defined modulo $\pi^{k+1}$.
     
  Let $\mathbf{g}$ denote the polynomial vector $(g_1,\dots,g_m)\in (\mathcal{O}_K[x_1,\dots,x_n])^m$. Consider the equation 
  $$\mathbf{g}(\mathbf{a}+\mathbf{\delta}\pi^k)\equiv \mathbf{g}(\mathbf{a})+\mathrm{J}(X)(\mathbf{a})\delta\pi^k \equiv 0 \pmod{\pi^{k+1}},$$
  where $\mathrm{J}(X)$ denotes the Jacobi matrix of $X$. Since $\mathbf{g}(\mathbf{a})\equiv 0\pmod{\pi^k}$, we may write  $\mathbf{g}(\mathbf{a})=\pi^k \mathbf{a}^{\prime}$ for some $\mathbf{a}^\prime\in (\mathcal{O}_K/(\pi))^n$. Then the equation above can be simplified as follows
  $$\mathbf{a}^{\prime }+\mathrm{J}(X)(\mathbf{a})\delta \equiv 0\pmod{\pi}.$$
  Since $X$ is smooth with $\mathrm{codim}X=d$ and has a good reduction at $\mathfrak{p}$,  we have $\mathrm{rank}(\mathrm{J}(X))=d$ over the residue field $\mathcal{O}_K/\mathfrak{p}$. Thus, there are $(\#({\mathcal{O}}_K/\mathfrak{p}))^{n-d}$ distinct solutions. Hence the cardinality of the set 
  $$\{\mathbf{b}\in X(\mathcal{O}_K/\mathfrak{p}^{k+1}): \mathbf{b}\equiv \mathbf{a}\pmod{\mathfrak{p}^k}\}$$
 is equal to $(\#({\mathcal{O}}_K/\pi))^{n-d}=(\#({\mathcal{O}}_K/\mathfrak{p}))^{n-d}=(\mathrm{N}_{K/\mathbb{Q}}(\mathfrak{p}))^{n-d}$.   
\end{proof}

We can now derive the explicit counting formula for prime powers by  inductively applying the lifting property established in Lemma \ref{lifing}.

\begin{proposition}\label{prime}
Let $X$ be a smooth closed subscheme of $\mathbb{A}^n_{\mathcal{O}_K}$ with $\mathrm{codim}X=d$, $\mathfrak{p}$ be a prime ideal of $\mathcal{O}_K$ such that $X$ has a good reduction at $\mathfrak{p}$. For a positive integer $e$, we have 
 $$\#\mathscr{E}_{f,\mathfrak{p}^e}(X)=(\mathrm{N}_{K/\mathbb{Q}}(\mathfrak{p}))^{(n-d)e}\left(\frac{\#X(\mathcal{O}_K/\mathfrak{p})-\# \mathcal{N}^f(\mathfrak{p},X)}{(\mathrm{N}_{K/\mathbb{Q}}(\mathfrak{p}))^{n-d}}\right).$$
\end{proposition}
\begin{proof}
    Consider the following exact sequence:
    $$0\longrightarrow \mathfrak{p}/\mathfrak{p}^e\longrightarrow \mathcal{O}_K/\mathfrak{p}^e\overset{\varphi}\longrightarrow \mathcal{O}_K/\mathfrak{p}\longrightarrow 0.$$
    The projection $\varphi$ induces a map 
    $$\phi: \mathscr{E}_{f,\mathfrak{p}^e}(X)\longrightarrow \mathscr{E}_{f,\mathfrak{p}}(X),\quad (x_1,\dots,x_n)\longmapsto (\varphi(x_1),\dots,\varphi(x_n))$$
 Since an element of $\mathcal O_K/\mathfrak p^e$ is a unit if and only if its image in $\mathcal O_K/\mathfrak p$ is a unit, the map $\phi$ is well defined.

Recall that $X$ is smooth with $\mathrm{codim}X=d$ and has a good reduction at $\mathfrak{p}$. By applying Lemma \ref{lifing} inductively, the cardinality of the fiber of the projection $X(\mathcal O_K/\mathfrak p^e)\longrightarrow X(\mathcal O_K/\mathfrak p)$
over any $\mathbf a\in X(\mathcal O_K/\mathfrak p)$
is equal to $(\mathrm N_{K/\mathbb Q}(\mathfrak p))^{(n-d)(e-1)}$.
 
     For any $\mathbf{a}\in X(\mathcal{O}_K/\mathfrak{p})$, there exist $(\mathrm{N}_{K/\mathbb{Q}}(\mathfrak{p}))^{(n-d)(e-1)}$ elements lifting $\mathbf{a}$ in $X(\mathcal{O}_K/\mathfrak{p}^e)$. Since the unit condition is preserved under reduction modulo $\mathfrak p$,
each element of $\mathscr E_{f,\mathfrak p}(X)$ admits exactly
$(\mathrm N_{K/\mathbb Q}(\mathfrak p))^{(n-d)(e-1)}$ lifts in
$\mathscr E_{f,\mathfrak p^e}(X)$.

     Hence, we have
     $$\#\mathscr{E}_{f,\mathfrak{p}^e}(X)=(\mathrm{N}_{K/\mathbb{Q}}(\mathfrak{p}))^{(n-d)(e-1)}\#\mathscr{E}_{f,\mathfrak{p}}(X),$$
 Since $\#\mathscr{E}_{f,\mathfrak{p}}(X)=\#X(\mathcal{O}_K/\mathfrak{p})-\# \mathcal{N}^f(\mathfrak{p},X)$, we have
 \begin{align*}
     \#\mathscr{E}_{f,\mathfrak{p}^e}(X)&=(\mathrm{N}_{K/\mathbb{Q}}(\mathfrak{p}))^{(n-d)(e-1)}(\#X(\mathcal{O}_K/\mathfrak{p})-\# \mathcal{N}^f(\mathfrak{p},X))\\
     &=(\mathrm{N}_{K/\mathbb{Q}}(\mathfrak{p}))^{(n-d)e}\left(\frac{\#X(\mathcal{O}_K/\mathfrak{p})-\# \mathcal{N}^f(\mathfrak{p},X)}{(\mathrm{N}_{K/\mathbb{Q}}(\mathfrak{p}))^{n-d}}\right).
\end{align*}
This completes the proof of Proposition \ref{prime}.
\end{proof}

With the explicit counting formula for prime powers established in Proposition \ref{prime}, we are ready to prove our main result.

\begin{proof}[Proof of Theorem \ref{main}]
    Let $\mathfrak{n}=\mathfrak{p}_1^{e_1}\dots \mathfrak{p}_g^{e_g}$ be the prime ideal factorization of $\mathfrak{n}$, where $\mathfrak{p}_1,\dots,\mathfrak{p}_g$ are distinct prime ideals of $\mathcal{O}_K$. According to Lemma\ref{decomposition}, we have $$ \#\mathscr{E}_{f,\mathfrak{n}}(X)=\prod_{i=1}^g\#\mathscr{E}_{f,\mathfrak{p}_i^{e_i}}(X).$$
    Using Proposition \ref{prime}, we have 
    \begin{align*}
        \#\mathscr{E}_{f,\mathfrak{n}}(X)&=\prod_{i=1}^g(\mathrm{N}_{K/\mathbb{Q}}(\mathfrak{p}_i))^{(n-d)e_i}\left(\frac{\#X(\mathcal{O}_K/\mathfrak{p}_i)-\# \mathcal{N}^f(\mathfrak{p}_i,X)}{(\mathrm{N}_{K/\mathbb{Q}}(\mathfrak{p}_i))^{n-d}}\right)\\
        &=(\mathrm{N}_{K/\mathbb{Q}}(\mathfrak{n}))^{n-d}\prod_{\mathfrak{p}\vert \mathfrak{n}}\left(\frac{\#X(\mathcal{O}_K/\mathfrak{p})-\# \mathcal{N}^f(\mathfrak{p},X)}{(\mathrm{N}_{K/\mathbb{Q}}(\mathfrak{p}))^{n-d}}\right).
    \end{align*}
    This completes the proof of Theorem \ref{main}.
\end{proof}

Having established the general counting formula, we now present a concrete example to show its application. In the following example, we explicitly compute the various local quantities involved and conclude with the precise count of $f$-exunits modulo $\mathfrak{n}$ for a specific variety.

\begin{example}
\label{e.g.}
    Let $K=\mathbb{Q}(\sqrt{-5})$, so that its ring of integers is $\mathcal{O}_K=\mathbb{Z}[\sqrt{-5}]$. We consider the variety $X\subset\mathbb{A}^2_{\mathcal{O}_K}$ defined by the circle equation $x^2+y^2=c$, where the parameter $c\in \mathbb{Z}\setminus \{0\}$. This variety $X$ is smooth and geometrically irreducible. Let $f(x)=x-a$ be a polynomial in $ \mathbb{Z}[x]$.
    
    To ensure good reductions, we assume the ideal $\mathfrak{n}$ is coprime to $(2c)$. Let $\mathfrak{p}$ be a prime ideal of $\mathcal{O}_K$ lying above a prime $p\in \mathbb{Z}$. The number of rational points on $X$ over the residue field $\mathcal{O}_K/\mathfrak{p}$ is given by the well-known formula $$\#X(\mathcal{O}_K/\mathfrak{p})=\mathrm{N}_{K/\mathbb{Q}}(\mathfrak{p})-\left(\frac{-1}{\mathfrak{p}}\right),$$
    where $\left(\frac{\cdot}{\mathfrak{p}}\right)$ is the generalized Legendre symbol.
    
    Next, we calculate intersection number $\#\mathcal{N}^f(\mathfrak{p},X)$. The set $\mathcal{N}^f(\mathfrak{p},X)$ consists of the points in $X(\mathcal{O}/\mathfrak{p})$ satisfying $x=a$ or $y=a$. If $c-a^2$ is not a square in $\mathcal{O}_K/\mathfrak{p}$, then $\#\mathcal{N}^f(\mathfrak{p},X)=0$. If $c-a^2$ is a square in $\mathcal{O}_K/\mathfrak{p}$, then there exists $b\in \mathcal{O}_K$ such that $b^2\equiv c-a^2\pmod{\mathfrak{p}}$. The intersection set $\mathcal{N}^f(\mathfrak{p},X)$ contains at most for points: $(a,b),(a,-b),(b,a)$ and $(-b,a)$. 
    
    The exact cardinality depends on the overlapping of these points. If $a^2-c\equiv 0\pmod{\mathfrak{p}}$, then $b=-b$, implying that there are $2$ points in $\mathcal{N}^f(\mathfrak{p},X)$. If $2a^2-c\equiv 0\pmod{\mathfrak{p}}$, then $b=\pm a$, implying that there are $3$ points in $\mathcal{N}^f(\mathfrak{p},X)$. Otherwise, there are exactly 4 distinct points in $\mathcal{N}^f(\mathfrak{p},X)$. Note that for any rational integers $r$, $r\equiv 0\pmod{\mathfrak{p}}$ if and only if $r\equiv 0\pmod{p}$. We define a function $m(p)$ on rational primes as follows
    
$$m(p)=
\begin{cases}
    2 \quad\text{ if }p\vert (c-a^2),\\
    3 \quad\text{ if }p\vert (c-2a^2),\\
    4 \quad\text{ if }\mathrm{gcd}(p,(c-2a^2)(c-a^2))=1.
\end{cases}$$
Consequently, we have $\#\mathcal{N}^f(\mathfrak{p},X)=m(p)$ if $c-a^2$ is a square in $\mathcal{O}_K/\mathfrak{p}$, and 0 otherwise.
    
The quadratic character of $c-a^2$ in the residue field $\mathcal{O}_K/\mathfrak{p}$ depends on the decomposition behavior of the rational prime $p$ in $K$. If $p$ is inert, then $\mathcal{O}_K/\mathfrak{p}=\mathbb{F}_{p^2}$. Since every element of $\mathbb{F}_p$ becomes a square in its quadratic extension $\mathbb{F}_{p^2}$, the integer $c-a^2$ must be a square in $\mathcal{O}_K/\mathfrak{p}$. If $p$ is splitting or ramified, then $\mathcal{O}_K/\mathfrak{p}=\mathbb{F}_{p}$. In this case, whether $c-a^2$ is a square in $\mathcal{O}_K/\mathfrak{p}$ depends on the Legendre symbol $\left(\frac{c-a^2}{p}\right)$.
 
Using the Dedekind-Kummer theorem \cite[Chapter 1, Proposition 8.3]{Neukirch}, the decomposition of a rational prime $p$ in $\mathcal{O}_K=\mathbb{Z}[\sqrt{-5}]$ is determined by the Legendre symbol $\left(\frac{-5}{p}\right)$. The prime $p$ splits in $\mathbb{Z}[\sqrt{-5}]$ if and only if $\left(\frac{-5}{p}\right)=1$, i.e. $p\equiv 1,3,7,9\pmod{20}$. The prime $p$ inerts in $\mathbb{Z}[\sqrt{-5}]$ if and only if $\left(\frac{-5}{p}\right)=-1$, i.e. $p\equiv 11,13,17,19\pmod{20}$. The discriminant of the field $\mathbb{Q}[\sqrt{-5}]$ is $d_K=-20$. Thus, a prime $p$ is ramified if and only if $p\vert 20$. Since we have assume that $p$ is coprime to $2$, the only possible ramified prime is $p=5$.

According to Theorem \ref{main}, the counting formular is given by
    $$\begin{aligned}
\#\mathcal{E}_{f,\mathfrak{n}}(X) =\mathrm{N}_{K/\mathbb{Q}}(\mathfrak{n}) & \prod_{p \vert N(\mathfrak{n}), \text{ inert}} \left(1 - \frac{1+m(p)}{p^2}\right) \\
& \times \prod_{p \vert N(\mathfrak{n}), \text{ split}} \left(1 - \frac{\left(\frac{-1}{p}\right) +\frac{1}{2}\left(\left(\frac{c-a^2}{p}\right)+1 \right)m(p)}{p}\right)^{\nu_{\mathfrak{n}}(p)} \\
& \times \prod_{p \vert N(\mathfrak{n}), \text{ ramified}} \left(1 - \frac{1+\frac{1}{2}\left(\left(\frac{c-a^2}{p}\right)+1 \right)m(p)}{p}\right),
\end{aligned}$$
    where ${\nu_\mathfrak{n}}$ is defined as follows
    $${\nu_\mathfrak{n}(p)}=
    \begin{cases}
        1 &\text{ if }(p)\nmid\mathfrak{n},\\
        2 &\text{ if }(p)\mid\mathfrak{n}.
    \end{cases}$$
\end{example}

As illustrated in Example \ref{e.g.}, the explicit formula of $\#\mathcal{E}_{f,\mathfrak{n}}(X)$ relies heavily on the arithmetic properties of individual primes, such as their splitting behavior in $K$ and the quadratic character of the polynomial values. Consequently, the local intersection term $\#\mathcal{N}^f(\mathfrak{p},X)$ varies irregularly, which makes the exact formula too complex to analyze the global growth rate of the counting function. Motivated by this observation, we turn our attention in the next section to the asymptotic behavior of $\#\mathcal{E}_{f,\mathfrak{n}}(X)$. By employing the Lang-Weil inequality, we can effectively control these local factors and establish a general bound that depends primarily on the geometric invariants of the variety $X$, rather than the specific arithmetic of each prime.

\section{Asymptotic estimates for the counting function}
\label{proof of approximation}
In this section, we use the Lang-Weil inequality to give an approximation of $\#\mathscr{E}_{f,\mathfrak{n}}(X)$ when $X$ is geometrically irreducible over $K$ and all primes dividing $\mathfrak{n}$ are of good reduction for $X$. We first state the Lang–Weil inequality, which will be used repeatedly in this section.
\begin{lemma}[\cite{LW} Theorem 1 or \cite{SG} Remark 11.3]\label{LW}
    Let $k=\mathbb{F}_q$ be a finite field, and X be a subvariety of $\mathbb{A}^n_k$ with $\mathrm{dim}X=r$ and  $\mathrm{deg}X=\ell$. Then we have
    $$\vert \#X(k)-q^r\vert =(\ell-1)(\ell-2)q^{r-\frac{1}{2}}+O(q^{r-1}).$$
    The implied constant in the $O$-terms depends only on $n,d,r$.
\end{lemma}

Using this result, we prove Theorem \ref{approximate} as follows.
\begin{proof}[Proof of Theorem \ref{approximate}]
    According to Theorem \ref{main}, we have 
    $$\#\mathscr{E}_{f,\mathfrak{n}}(X)=(\mathrm{N}_{K/\mathbb{Q}}(\mathfrak{n}))^{n-d}\prod_{\mathfrak{p}\vert\mathfrak{n}}\left(\frac{\#X(\mathcal{O}_K/\mathfrak{p})-\# \mathcal{N}^f(\mathfrak{p},X)}{(\mathrm{N}_{K/\mathbb{Q}}(\mathfrak{p}))^{n-d}}\right).$$ 
    Recall that $X$ is smooth and geometrically irreducible, and has a good reduction at $\mathfrak{p}$ for any prime $\mathfrak{p}\vert\mathfrak{n}$. Using the Lang-Weil inequality, we obtain that $$\#X(\mathcal{O}_K/\mathfrak{p})=\mathrm{N}_{K/\mathbb{Q}}(\mathfrak{p})^{n-d}+O(\mathrm{N}_{K/\mathbb{Q}}(\mathfrak{p})^{n-d-\frac{1}{2}}).$$
    
    Let $V(f)$ denote the set of zero points of the polynomial $f$. Since $\mathcal{N}^f(\mathfrak{p},X)=X(\mathcal{O}_K/\mathfrak{p})\cap \bigcup_{i=1}^nV(f(x_i))$, if there exists a prime ideal $\mathfrak{p}\vert\mathfrak{n}$ such that $X(\mathcal{O}_K/\mathfrak{p})\subset \mathcal{N}^f(\mathfrak{p},X)$, then there exist no $(f,\mathfrak{p})$-exunits on $X(\mathcal{O}_K/\mathfrak{p})$. By Theorem\ref{main}, we have the trivial case $\#\mathscr{E}_{f,\mathfrak{n}}(X)=0$. If $\mathcal{N}^f(\mathfrak{p},X)$ is a proper subset of $X(\mathcal{O}_K/\mathfrak{p})$, then each irreducible component of $\mathcal{N}^f(\mathfrak{p},X)$ has dimension at most $n-d-1$. On the other hand, using the generalized Bézout's Theorem in the intersection theory, we find that the number of irreducible components of $\mathcal{N}^f(\mathfrak{p},X)$ is uniformly bounded by constant depending on $d,n$ and the degree of $f$. Hence, using the Lang-Weil inequality, we obtain that $$\#\mathcal{N}^f(\mathfrak{p},X)=O(\mathrm{N}_{K/\mathbb{Q}}(\mathfrak{p})^{n-d-1}).$$
    The implied constant in the $O$-terms above depends only on $K,X$ and $f$.
    
    As a result, we have 
    \begin{align*}
        \#\mathscr{E}_{f,\mathfrak{n}}(X)&=(\mathrm{N}_{K/\mathbb{Q}}(\mathfrak{n}))^{n-d}\prod_{\mathfrak{p}\vert \mathfrak{n}}\left(\frac{\#X(\mathcal{O}_K/\mathfrak{p})-\# \mathcal{N}^f(\mathfrak{p},X)}{(\mathrm{N}_{K/\mathbb{Q}}(\mathfrak{p}))^{n-d}}\right)\\
        &=(\mathrm{N}_{K/\mathbb{Q}}(\mathfrak{n}))^{n-d}\prod_{\mathfrak{p}\vert \mathfrak{n}}\left(1+O(\mathrm{N}_{K/\mathbb{Q}}(\mathfrak{p})^{-\frac{1}{2}})\right)\\
        &=(\mathrm{N}_{K/\mathbb{Q}}(\mathfrak{n}))^{n-d}\exp\left({O\left(\sum_{\mathfrak{p}\vert\mathfrak{n}}\mathrm{N}_{K/\mathbb{Q}}(\mathfrak{p})^{-\frac{1}{2}}\right)}\right).
    \end{align*}
    We can bound the sum in the error term as follows:
    $$\sum_{\mathfrak{p}\vert\mathfrak{n}}\mathrm{N}_{K/\mathbb{Q}}(\mathfrak{p})^{-1/2}\le \omega(\mathfrak{n})=O\left(\frac{\log\mathrm{N}_{K/\mathbb{Q}}(\mathfrak{n})}{\log\log\mathrm{N}_{K/\mathbb{Q}}(\mathfrak{n})}\right),$$
    where $\omega(\mathfrak{n})$ is the number of prime ideals of $\mathcal{O}_K$ dividing $\mathfrak{n}$. 
    
    Consequently, we obtain the estimate
    $$\#\mathscr{E}_{f,\mathfrak{n}}(X)=(\mathrm{N}_{K/\mathbb{Q}}(\mathfrak{n}))^{n-d}\exp\left({O\left(\frac{\log\mathrm{N}_{K/\mathbb{Q}}(\mathfrak{n})}{\log\log\mathrm{N}_{K/\mathbb{Q}}(\mathfrak{n})}\right)}\right).$$
    This completes the proof of Theorem \ref{approximate}.
\end{proof}

Lemma \ref{LW} implies that the coefficient of the dominant error term $q^{r-\frac{1}{2}}$ vanishes whenever $\deg(X) \le 2$. This observation allows us to derive the sharper estimate presented in Theorem \ref{approximate 2}.

\begin{proof}[Proof of Theorem \ref{approximate 2}]
    We adopt the same notation and setup as in the proof of Theorem \ref{approximate}. Since $\mathrm{deg}(X) \le 2$, the coefficient $(\mathrm{deg}X-1)(\mathrm{deg}X-2)$ in Lemma \ref{LW} is zero. Consequently, the estimate for the number of rational points on $X$ improves to
    $$\#X(\mathcal{O}_K/\mathfrak{p}) = \mathrm{N}_{K/\mathbb{Q}}(\mathfrak{p})^{n-d} + O(\mathrm{N}_{K/\mathbb{Q}}(\mathfrak{p})^{n-d-1}).$$
    
    Regarding the intersection with the zero set of $f$, we retain the bound established in the previous proof
    $$\#\mathcal{N}^f(\mathfrak{p},X)=O(\mathrm{N}_{K/\mathbb{Q}}(\mathfrak{p})^{n-d-1}).$$
    
    Combining these two estimates, we have
    \begin{align*}
        \#\mathscr{E}_{f,\mathfrak{n}}(X)&=(\mathrm{N}_{K/\mathbb{Q}}(\mathfrak{n}))^{n-d}\prod_{\mathfrak{p}\vert \mathfrak{n}}\left(\frac{\#X(\mathcal{O}_K/\mathfrak{p})-\# \mathcal{N}^f(\mathfrak{p},X)}{(\mathrm{N}_{K/\mathbb{Q}}(\mathfrak{p}))^{n-d}}\right)\\
        &=(\mathrm{N}_{K/\mathbb{Q}}(\mathfrak{n}))^{n-d}\prod_{\mathfrak{p}\vert \mathfrak{n}}\left(1+O(\mathrm{N}_{K/\mathbb{Q}}(\mathfrak{p})^{-1})\right)\\
        &=(\mathrm{N}_{K/\mathbb{Q}}(\mathfrak{n}))^{n-d}\exp\left({O(\sum_{\mathfrak{p}\vert\mathfrak{n}}\mathrm{N}_{K/\mathbb{Q}}(\mathfrak{p})^{-1})}\right).
    \end{align*}
    Since $\sum_{\mathfrak{p}\vert\mathfrak{n}}\mathrm{N}_{K/\mathbb{Q}}(\mathfrak{p})^{-1}=O(\log\log(\mathrm{N}_{K/\mathbb{Q}}(\mathfrak{n})))$, we have \begin{align*}
    \#\mathscr{E}_{f,\mathfrak{n}}(X)&=\mathrm{N}_{K/\mathbb{Q}}(\mathfrak{n})^{n-d}\exp\left(O(\log\log \mathrm{N}_{K/\mathbb{Q}}(\mathfrak{n}))\right)\\
    &=\mathrm{N}_{K/\mathbb{Q}}(\mathfrak{n})^{n-d}\left(\log \mathrm{N}_{K/\mathbb{Q}}(\mathfrak{n})\right)^{O(1)}.
    \end{align*}
This completes the proof of Theorem \ref{approximate 2}.
\end{proof}
\section{Concluding Remarks}
In this paper, we have established exact formulas and asymptotic estimates for the number of $f$-exunits on smooth algebraic varieties, assuming good reduction at all relevant primes. 

A natural direction for future research is to extend these results to singular varieties or to analyze the distribution of solutions at primes of bad reduction. In such cases, the geometry of the fibers becomes significantly more complex, and the error terms in the Lang-Weil inequality may behave differently.

Furthermore, the similarity between number fields and function fields suggests that one can generalize these results to the function fields. 

Another interesting direction is to study the dynamical properties of the counting function. Harrington and Jones \cite{HJ} studied the iteration of $\phi_e(n)$ in $\mathbb{Z}_n$, and Miguel \cite{Miguel} suggested extending this to the ring of integers $\mathcal{O}_K$. However, defining a natural iteration for general $f$-exunits is complicated by the varying reduction types.

\section*{Acknowledgements}
The authors are grateful to Zhengyu Tao for helpful discussions at the early stage of this work and for his valuable suggestions on the manuscript.

\end{document}